\documentclass[10pt,a4paper]{article}
\usepackage[top=3cm, bottom=3cm, left=3cm, right=3cm]{geometry}
\usepackage[latin1]{inputenc}
\usepackage{amsmath}
\usepackage{amsthm}
\usepackage{amsfonts}
\usepackage{amssymb}
\usepackage{graphics}
\usepackage{tabu}
\usepackage{breqn}
\usepackage[hang,flushmargin]{footmisc} 

\usepackage{amsmath,amsthm,amssymb,graphicx, multicol, array}
\usepackage{enumerate}
\usepackage{enumitem}
\usepackage{hyperref}
\usepackage{setspace}
\usepackage{xcolor}
\usepackage{currfile}

\newtheorem{thm}{Theorem}[section]
\newtheorem{lem}{Lemma}[section]
\newtheorem{conj}{Conjecture}[section]

\newtheorem{cor}{Corollary}[section]
\newtheorem{dfn}{Definition}[section]
\newtheorem{exe}{Exercise}[section]

\newcommand{\N}{\mathbb{N}}
\newcommand{\Z}{\mathbb{Z}}
\newcommand{\Q}{\mathbb{Q}}
\newcommand{\R}{\mathbb{R}}
\newcommand{\C}{\mathbb{C}}

\usepackage{fancyhdr}
\title{The Zeta Quotient $\zeta(3)/ \pi^3$ is Irrational}
\date{}
\author{N. A. Carella}

\begin{document}
\thispagestyle{empty}
\date{}
\maketitle

\vskip .25 in

\textbf{\textit{Abstract}:} This note proves that the first odd zeta value does not have a closed form formula $\zeta(3)\ne r \pi^3$ for any rational number $r \in \Q$. Furthermore, assuming the irrationality of the second odd zeta value $\zeta(5)$, it is shown that $\zeta(5)\ne r \pi^5$ for any rational number $r \in \Q$. \let\thefootnote\relax\footnote{ \today \date{} 
}

\vskip .25 in 

\maketitle
\section{Introduction}\label{s4499}
The first even zeta value has a closed form formula $\zeta(2)=\pi^2/6$, and the first odd zeta value has a nearly closed formula given by the Lerch representation
\begin{equation}\label{eq4499.02}
 \zeta(3)=\frac{7 \pi^3}{180}  -2 \sum_{n \geq 1} \frac{1}{n^3( e^{2 \pi n}-1)},
\end{equation}
confer \eqref{eq195.23} for more general version. This representation is a special case of the Ramanujan formula for odd zeta values, see \cite{BS17}, and other generalized version proved in \cite{GE72}. However, it is not known if there exists a closed form formula $\zeta(3)=r \pi^3$ with $r \in \Q$, see \cite[p.\ 167]{WM05}, \cite[p.\ 3]{LR11} for related materials. 
\begin{thm}\label{thm4499.30} The odd zeta value 
\begin{equation}\label{eq4499.04}
\zeta(3)\ne r\pi^3,
\end{equation}
for any rational number $r \in \Q$.
\end{thm}
The proof of this result is simply a corollary of Theorem \ref{thm599.53} in Section \ref{s599}. The basic technique generalizes to other odd zeta values. As an illustration of the versatility of this technique, conditional on the irrationality of the zeta value $\zeta(5)$, in Theorem \ref{thm519.50} it is shown that $\zeta(5)/\pi^5$ is irrational. The preliminary Section \ref{s799} and Section \ref{s719}  develop the required foundation, and the proofs of the various results are presented in Section \ref{s599} and Section \ref{s519}.  
\section{Basic Foundation For $2n+1=3$}\label{s799}
The results for the nonvanishing of the sine function at certain real numbers can be derived by several different methods such as Weil criterion for uniformly distributed sequences, evaluations of infinite products, and other techniques. Some of these ideas are considered here. \\

The case is equivalent to the Weil criterion for uniformly distributed sequences, confer \cite[p.\ 8, Theorem 2.1]{KN74} for more details. 

\begin{lem}\label{lem799.10} If $k \in \mathbb{Z}^{\times}$ is an integer, then 
\begin{multicols}{2}
\begin{enumerate} [font=\normalfont, label=(\roman*)]
\item$ \sin(\zeta(3)k)\ne 0$
\item$ \sin(\zeta(3)^{-1} k)\ne 0$.
\end{enumerate}
\end{multicols}
\end{lem}
\begin{proof} (i) Fix an integer $k \ne0$, and consider the uniformly distributed sequence $\{2 \zeta(3)kn:n \geq 1\}$. From the uniform distribution of this sequence, \cite[Theorem 4.2]{RH94}, it immediately follows that $\zeta(3)k\ne m\pi$ for all $k \in \Z$. Applying Lemma \ref{lem799.72} yields
\begin{eqnarray}\label{eq799.10}
\frac{1}{x}\sum_{-x\leq n \leq x}e^{i2  \zeta(3)kn}&=&\frac{1}{x} \frac{\sin\left(  (2x+1)\zeta(3)k\right )}{\sin\left( \zeta(3) k\right )}\\
&=&o(1) \nonumber.
\end{eqnarray}
This implies that $\sin(\zeta(3)k)\ne 0$ as required. (ii) Since $\zeta(3)$ is irrational, see Theorem \ref{thm599.01}, statement (i) implies statement (ii).
\end{proof}

Unfortunately, the validity of the opposite result $\sin(\zeta(3)k)= 0$ is false. If it was true, it would lead to a far more rewarding result, that is, $\zeta(3)=r \pi$ for some $r \in \Q^{\times}$. To see this, write  
$\zeta(3)=7\pi^3/180-v$, where $v \in \R^{\times}$, see \eqref{eq4499.02}. Then,
\begin{eqnarray}\label{eq799.15}
0&=&\sin(\zeta(3)k)\\
&=&\sin(7\pi^3k/180-vk)\nonumber\\
&=&\sin(7\pi^3k/180)\cos(vk)-\sin(vk)\cos(7\pi^3k/180).\nonumber\\
\end{eqnarray}
Equivalently $\tan(7\pi^3k/180) =\tan(vk)$. The tangent function is periodic of period $\pi$, one-to-one, and monotonically increasing over the interval $[-\pi/2,\pi/2]$. Hence, the two arguments satisfy the equation 
\begin{equation}\label{eq799.16}
\frac{7 \pi^3}{180}k=vk+m\pi
\end{equation}
where $m \in \Z$. Lastly, $\zeta(3)=(m/k)\pi$.
\begin{lem}\label{lem799.20} If $k \in \mathbb{Z}^{\times}$ is an integer, then 
\begin{multicols}{2}
\begin{enumerate} [font=\normalfont, label=(\roman*)]
\item$ \sin(\zeta(3)\pi^{-1}k)\ne 0$,
\item$ \sin(\zeta(3)^{-1}\pi k)\ne 0$.
\end{enumerate}
\end{multicols}

\end{lem}
\begin{proof} (ii) Evaluate the product representation of the sine function 
\begin{eqnarray}\label{eq799.82}
\sin \left ( \frac{\pi}{\zeta(3)}k\right ) &=&\frac{\pi k}{\zeta(3)}\prod_{n \geq1} \left (1- \frac{(\pi k/\zeta(3))^2}{\pi^2n^2} \right ) \nonumber \\
&=&\frac{\pi k}{\zeta(3)} \prod_{n \geq1} \left (1- \frac{k^2}{\zeta(3)^2n^2} \right ) \\
&\ne&0\nonumber.
\end{eqnarray} 
The last line is valid for any irrational value $\zeta(s)$, $s\geq 3$, see Theorem \ref{thm599.01}. (i) The proof is similar.
\end{proof}

\begin{lem}\label{lem799.30} If $k \in \mathbb{Z}^{\times}$ is an integer, then
\begin{multicols}{2}
 \begin{enumerate} [font=\normalfont, label=(\roman*)]
\item$\sin(\zeta(3)^{-1}\pi^2 k)\ne 0$,
\item$ \sin(\zeta(3)\pi^{-2} k)\ne 0$.
\end{enumerate}
\end{multicols}
\end{lem}
\begin{proof} (i) Fix an integer $k \geq 1$, let $\{u_n/v_n: n\geq 1\}$ be the sequence of convergents of the irrational number $\pi/\zeta(3) $, refer to Theorem \ref{thm599.10}, and let \begin{equation}\label{eq799.55}\left | \pi \left ( \frac{\pi}{\zeta(3)}k- m \right )\right |<1,
\end{equation} 
where $m \in \mathbb{N}$ is an integer. Then, as $n \to \infty$, 
\begin{eqnarray}\label{eq799.82}
\left |\sin \left ( \frac{\pi^2}{\zeta(3)}k \right ) \right |&=&\left |\sin \left ( \pi \left ( \frac{\pi}{\zeta(3)}k- m \right )\right )\right | \nonumber \\
&\geq&\left |\sin \left ( \pi \left ( \frac{\pi}{\zeta(3)}v_n- u_n \right )\right )\right |.
\end{eqnarray} 
This follows from the Dirichlet approximation theorem
\begin{equation}\label{eq799.54}
\left |  \frac{\pi}{\zeta(3)}v_n- u_n \right |  \leq \left |  \frac{\pi}{\zeta(3)}k- m \right |
\end{equation}
for any integer $k \leq v_n$, see \cite[Theorem 2.1]{RH94} for the best rational approximations. Therefore, 
\begin{eqnarray}\label{eq799.84}
\left |\sin \left ( \frac{\pi}{\zeta(3)}\right )\right |
&\geq&\left |\sin \left ( \pi \left ( \frac{\pi}{\zeta(3)}v_n- u_n \right )\right )\right |\\
&\gg&\left |  \frac{\pi}{\zeta(3)}v_n- u_n \right |\nonumber\\
&\gg&  \frac{1}{v_{n+1}}\nonumber\\
&\ne& 0 \nonumber.
\end{eqnarray}
The third line in \eqref{eq799.84} follows from the basic Diophantine inequality
\begin{equation}\label{eq799.86}
\frac{1}{2q_{n+1}}\leq \left |  \alpha q_n- p_n \right |  \leq \frac{1}{q_n}
\end{equation}
for any irrational number $\alpha \in \R$, confer \cite[Theorem 3.8]{OC63}, \cite[Theorem 13]{KA97}, and similar references.
\end{proof}

\begin{lem}\label{lem799.72} For any real number $t \ne k \pi$ with $k \in \mathbb{Z}$, and a large integer $x \geq 1$, the finite sum
\begin{equation}\label{eq799.50}
\sum_{-x\leq n \leq x}e^{i2 tn}=\frac{\sin((2x+1)t)}{\sin(t)}.
\end{equation}
\end{lem}
\begin{proof} Expand the complex exponential sum into two subsums:
\begin{equation}\label{eq799.50}
\sum_{-x\leq n \leq x}e^{i2 tn}=e^{-i2  t}\sum_{0\leq n \leq x-1}e^{-i2  tn}+\sum_{0\leq n \leq x}e^{i2 tn}.
\end{equation}
Lastly, use the geometric series to determine the closed form.	
\end{proof}
\section{Basic Foundation For $2n+1=5$}\label{s719}
The verification of $\sin \alpha \ne 0$, using a single method, but restricted to the zeta quotients $\alpha=\zeta(5)/\pi^{a}$, $a \leq 3$ is given below.
\begin{lem}\label{lem719.25} Let $k \in \mathbb{Z}^{\times}$ be an integer, and assume that $\zeta(5)\pi^{-5}=u/v\in\mathbb{Q}^{\times}$ is rational, where $v>1$, and $a \geq 1$. Then $\sin(\zeta(5)\pi^{-a}k)\ne 0$ for $a \leq 3$.
\end{lem}
\begin{proof} The evaluation of the product representation of the sine function returns
\begin{eqnarray}\label{eq719.25}
\sin \left ( \frac{\zeta(5)}{\pi^{a}}k\right ) &=&\frac{\zeta(5)}{\pi^{a}}k \prod_{n \geq1} \left (1- \frac{(\zeta(5)\pi^{-a}k)^2}{\pi^2n^2} \right ) \nonumber \\
&=&\frac{\zeta(5)}{\pi^{a}}k\prod_{n \geq1} \left (1- \frac{\zeta(5)^2k^2}{\pi^{2+2a}n^2} \right ) \nonumber.
\end{eqnarray} 
The hypothesis $\zeta(5)/\pi^{5}=u/v$ means that
\begin{equation}\label{eq719.24}
1- \frac{\zeta(5)^2k^2}{\pi^{2+2a}n^2}=1- \frac{\pi^{8-2a}u^2k^2}{v^{2}n^2}\ne 0,
\end{equation}
for $a \leq 3$, and $0\ne k, n, u, v \in \Z$. Hence, 
\begin{equation}\label{eq719.25}
\sin \left ( \frac{\zeta(5)}{\pi^{a}}k\right ) =\frac{\zeta(5)}{\pi^{a}}k\prod_{n \geq1} \left (1- \frac{(\zeta(5)\pi^{-a}k)^2}{\pi^2n^2} \right )  \ne0
\end{equation} 
for $a \leq 3$ as claimed. 
\end{proof}

\begin{lem}\label{lem719.30} If $k \in \mathbb{Z}^{\times}$ is an integer, then
\begin{multicols}{2}
 \begin{enumerate} [font=\normalfont, label=(\roman*)]
\item$\sin(\zeta(5)^{-1}\pi^4 k)\ne 0$,
\item$ \sin(\zeta(5)\pi^{-4} k)\ne 0$.
\end{enumerate}
\end{multicols}
\end{lem}
\begin{proof} (i) Fix an integer $k \geq 1$, let $\{u_n/v_n: n\geq 1\}$ be the sequence of convergents of the irrational number $\pi^3/\zeta(5) $, refer to Theorem \ref{thm599.10}, and let \begin{equation}\label{eq799.55}\left | \pi \left ( \frac{\pi^3}{\zeta(5)}k- m \right )\right |<1,
\end{equation} 
where $m \in \mathbb{N}$ is an integer. Then, as $n \to \infty$, 
\begin{eqnarray}\label{eq719.82}
\left |\sin \left ( \frac{\pi^4}{\zeta(5)}k \right ) \right |&=&\left |\sin \left ( \pi \left ( \frac{\pi^3}{\zeta(5)}k- m \right )\right )\right | \nonumber \\
&\geq&\left |\sin \left ( \pi \left ( \frac{\pi^3}{\zeta(5)}v_n- u_n \right )\right )\right |.
\end{eqnarray} 
This follows from the Dirichlet approximation theorem
\begin{equation}\label{eq719.54}
\left |  \frac{\pi^3}{\zeta(5)}v_n- u_n \right |  \leq \left |  \frac{\pi^3}{\zeta(5)}k- m \right |
\end{equation}
for any integer $k \leq v_n$, see \cite[Theorem 2.1]{RH94} for the best rational approximations. Therefore, 
\begin{eqnarray}\label{eq719.84}
\left |\sin \left ( \frac{\pi^4}{\zeta(5)}\right )\right |
&\geq&\left |\sin \left ( \pi \left ( \frac{\pi^3}{\zeta(5)}v_n- u_n \right )\right )\right |\\
&\gg&\left |  \frac{\pi^3}{\zeta(5)}v_n- u_n \right |\nonumber\\
&\gg&  \frac{1}{v_{n+1}}\nonumber\\
&\ne& 0 \nonumber.
\end{eqnarray}
The third line in \eqref{eq719.84} follows from the basic Diophantine inequality
\begin{equation}\label{eq719.86}
\frac{1}{2q_{n+1}}\leq \left |  \alpha q_n- p_n \right |  \leq \frac{1}{q_n}
\end{equation}
for any irrational number $\alpha \in \R$, confer \cite[Theorem 3.8]{OC63}, \cite[Theorem 13]{KA97}, and similar references.
\end{proof}

\section{Basic Results For The Zeta Quotients $\zeta(3)/\pi^3$}\label{s599}
The third odd zeta value has a few irrationality proofs, confer \cite{AR79}, \cite{BF79}, \cite[Chapter 4]{SJ05}, and \cite{BL15}. 

 \begin{thm} {\normalfont (\cite{AR79})} \label{thm599.01} The real numbers 
\begin{multicols}{2}
\begin{enumerate} [font=\normalfont, label=(\roman*)]
\item$ \zeta(3)$,
\item$ 1/\zeta(3)$,
\end{enumerate}
\end{multicols}
are irrationals.
\end{thm}
The decimal expansion $\zeta(3)=1.2 0 2 0 5 6 9 0 3 1 5 9 5 9 4 2\ldots$, appears as sequence $A002117$ in \cite{OEIS}. The irrationality proofs for various associated real numbers  considered here are derived from the theory of equidistribution of sequences of real numbers, see \cite{KN74} and similar references.
 \begin{thm}\label{thm599.10} The real numbers 
\begin{multicols}{2}
\begin{enumerate} [font=\normalfont, label=(\roman*)]
\item$ \zeta(3)/\pi$,
\item$ \pi/\zeta(3)$,
\end{enumerate}
\end{multicols}
are irrationals.
\end{thm}

\begin{proof} (i) On the contrary the number $\zeta(3)/ \pi= A/B$ is a rational number, where $A, B \in\mathbb{N}$ are integers such that $\gcd(A,B)=1$. This implies that the two sequences
\begin{equation}\label{eq599.10}
\{2 A\pi \cdot n: n \geq 1\}=\{2 B\zeta(3) \cdot n: n \geq 1\}
\end{equation}
are equivalent.  Likewise, the two limits 
\begin{equation}\label{eq599.12}
\lim_{x \to \infty}\frac{1}{2x} \sum_{-x \leq n \leq x}e^{i2A \pi n}= \lim_{x \to \infty}\frac{1}{2x} \sum_{-x \leq n \leq x}e^{i2 B\zeta(3) n}
\end{equation}
are equivalent. These limits are evaluated in two distinct ways.\\

I. Based on the independent properties of the number $\pi$. Use the identity $e^{i2 A\pi n}=1$ to evaluate of the left side limit as
\begin{equation}\label{eq599.14}
\lim_{x \to \infty}\frac{1}{2x} \sum_{-x\leq n \leq x}e^{i2A \pi n}=\lim_{x \to \infty}\frac{1}{2x} \sum_{-x\leq n \leq x}1=1.
\end{equation}

II. Based on the independent properties of the number $\zeta(3)$. By Lemma \ref{lem799.10}, $\sin(t)=\sin(B\zeta(3))\ne 0$ for any integer $B\ne 0$. Applying Lemma \ref{lem799.72}, the right side has the limit  
\begin{eqnarray}\label{eq599.16}
\lim_{x \to \infty}\frac{1}{2x} \sum_{-x \leq n \leq x}e^{i2B\zeta(3) n}
&=& \lim_{x \to \infty}\frac{1}{2x}\frac{\sin((2x+1)t)}{\sin(t)} \nonumber \\
&\leq &\lim_{x \to \infty}\frac{1}{2x}\frac{1}{\left |\sin(B\zeta(3))\right |}\\
&=&0 \nonumber.
\end{eqnarray}
Clearly, these two distinct limits contradict both equation (\ref{eq599.10}) and equation (\ref{eq599.12}). Specifically,
\begin{equation}\label{eq599.18}
1=\lim_{x \to \infty}\frac{1}{2x} \sum_{-x \leq n \leq x}e^{i2A \pi n}\ne \lim_{x \to \infty}\frac{1}{2x} \sum_{-x \leq n \leq x}e^{i2 B\zeta(3) n}=0.
\end{equation}

Hence, the real number $\zeta(3)/\pi\ne A/B$ is an irrational number. (ii) This statement follows from statement (i).
\end{proof}
 \begin{thm}\label{thm599.20} The real numbers 
\begin{multicols}{2}
\begin{enumerate} [font=\normalfont, label=(\roman*)]
\item$ \zeta(3)/\pi^2$,
\item$ \pi^2/\zeta(3)$,
\end{enumerate}
\end{multicols}
are irrationals.
\end{thm}

\begin{proof} (i) On the contrary the number $ \zeta(3)/\pi^2= A/B$ is a rational number, where $A, B \in\mathbb{N}$ are integers such that $\gcd(A,B)=1$. This implies that the two sequences
\begin{equation}\label{eq999.10}
\{2 A\pi \cdot n: n \geq 1\}=\{2 B\zeta(3)\pi^{-1} \cdot n: n \geq 1\}
\end{equation}
are equivalent.  Likewise, the two limits 
\begin{equation}\label{eq999.12}
\lim_{x \to \infty}\frac{1}{2x} \sum_{-x \leq n \leq x}e^{i2A \pi n}= \lim_{x \to \infty}\frac{1}{2x} \sum_{-x \leq n \leq x}e^{i2 B\zeta(3)\pi^{-1} n}
\end{equation}
are equivalent. These limits are evaluated in two distinct ways.\\

I. Based on the independent properties of the number $\pi$. Use the identity $e^{i2 A\pi n}=1$ to evaluate of the left side limit as
\begin{equation}\label{eq999.14}
\lim_{x \to \infty}\frac{1}{2x} \sum_{-x\leq n \leq x}e^{i2A \pi n}=\lim_{x \to \infty}\frac{1}{2x} \sum_{-x\leq n \leq x}1=1.
\end{equation}

II. Based on the independent properties of the number $\zeta(3)\pi^{-1}$. By Lemma \ref{lem799.20}, $\sin(t)=\sin(B\zeta(3)\pi^{-1})\ne 0$ for any integer $B\ne 0$. Applying Lemma \ref{lem799.72}, the right side has the limit  
\begin{eqnarray}\label{eq999.16}
\lim_{x \to \infty}\frac{1}{2x} \sum_{-x \leq n \leq x}e^{i2B\zeta(3)\pi^{-1} n}
&=& \lim_{x \to \infty}\frac{1}{2x}\frac{\sin((2x+1)t)}{\sin(t)} \nonumber \\
&\leq &\lim_{x \to \infty}\frac{1}{2x}\frac{1}{\left |\sin(B\zeta(3)\pi^{-1})\right |}\\
&=&0 \nonumber.
\end{eqnarray}
Clearly, these two distinct limits contradict both equation (\ref{eq999.10}) and equation (\ref{eq999.12}). Specifically,
\begin{equation}\label{eq599.52}
1=\lim_{x \to \infty}\frac{1}{2x} \sum_{-x \leq n \leq x}e^{i2A \pi n}\ne \lim_{x \to \infty}\frac{1}{2x} \sum_{-x \leq n \leq x}e^{i2 B\zeta(3)\pi^{-1} n}=0.
\end{equation}

Hence, the real number $\zeta(3)/\pi^2\ne A/B$ is an irrational number. (ii) This statement follows from statement (i).
\end{proof}

\begin{thm}\label{thm599.53} The real numbers 
\begin{multicols}{2}
\begin{enumerate} [font=\normalfont, label=(\roman*)]
\item$ \zeta(3)/\pi^3$,
\item$ \pi^3/\zeta(3)$,
\end{enumerate}
\end{multicols}
are irrationals.
\end{thm}

\begin{proof} (i) On the contrary the number $ \zeta(3)/\pi^3= A/B$ is a rational number, where $A, B \in\mathbb{N}$ are integers such that $\gcd(A,B)=1$. This implies that the two sequences
\begin{equation}\label{eq599.55}
\{2 A\pi \cdot n: n \geq 1\}=\{2 B\zeta(3)\pi^{-2} \cdot n: n \geq 1\}
\end{equation}
are equivalent.  Likewise, the two limits 
\begin{equation}\label{eq599.57}
\lim_{x \to \infty}\frac{1}{2x} \sum_{-x \leq n \leq x}e^{i2A \pi n}= \lim_{x \to \infty}\frac{1}{2x} \sum_{-x \leq n \leq x}e^{i2 B\zeta(3)\pi^{-2} n}
\end{equation}
are equivalent. These limits are evaluated in two distinct ways.\\

I. Based on the independent properties of the number $\pi$. Use the identity $e^{i2 A\pi n}=1$ to evaluate of the left side limit as
\begin{equation}\label{e599.44}
\lim_{x \to \infty}\frac{1}{2x} \sum_{-x\leq n \leq x}e^{i2A \pi n}=\lim_{x \to \infty}\frac{1}{2x} \sum_{-x\leq n \leq x}1=1.
\end{equation}

II. Based on the independent properties of the number $\zeta(3)\pi^{-2}$. By Lemma \ref{lem799.30}, $\sin(t)=\sin(B\zeta(3)\pi^{-2})\ne 0$ for any integer $B\ne 0$. Applying Lemma \ref{lem799.72}, the right side has the limit  
\begin{eqnarray}\label{eq599.36}
\lim_{x \to \infty}\frac{1}{2x} \sum_{-x \leq n \leq x}e^{i2B\zeta(3)\pi^{-2} n}
&=& \lim_{x \to \infty}\frac{1}{2x}\frac{\sin((2x+1)t)}{\sin(t)} \nonumber \\
&\leq &\lim_{x \to \infty}\frac{1}{2x}\frac{1}{\left |\sin(B\zeta(3)\pi^{-2})\right |}\\
&=&0 \nonumber.
\end{eqnarray}
Clearly, these two distinct limits contradict both equation (\ref{eq599.55}) and equation (\ref{eq599.57}). Specifically,
\begin{equation}\label{eq599.77}
1=\lim_{x \to \infty}\frac{1}{2x} \sum_{-x \leq n \leq x}e^{i2A \pi n}\ne \lim_{x \to \infty}\frac{1}{2x} \sum_{-x \leq n \leq x}e^{i2 B\zeta(3)\pi^{-2} n}=0.
\end{equation}
Hence, the real number $\zeta(3)/\pi^3\ne A/B$ is an irrational number. (ii) This statement follows from statement (i).
\end{proof}

\begin{cor}\label{cor599.13} For any integer $k \in \Z$, the real numbers $ \pi^k$  and $ \zeta(3)$ are linearly independent over the rational numbers $\mathbb{Q}$.
\end{cor}
\begin{proof} Without loss in generality let $k \geq 1$, and assume that the  equation
\begin{equation}
a\pi^k +b\zeta(3)=0
\end{equation}
has a nontrivial rational solution $a,b \in\mathbb{Q}^{\times}$. Proceed to employ the same technique as in the proof of Theorem \ref{thm599.10} to complete the proof.
\end{proof}

\begin{cor}\label{cor599.15} The real number
\begin{equation}
-\zeta^{\prime}(-2)=\frac{\zeta(3)}{4\pi^2}
\end{equation}
is irrational.
\end{cor}
\begin{proof} The functional equation of the zeta function provides an analytic continuation to the entire complex plane plane, see \cite[Theorem 1.6]{IV03}. Thus, its derivative is
\begin{eqnarray}\label{eq599.39}
\zeta^{\prime}(s)
&=&\frac{d}{ds}\left (  2 \pi^{s-1} \sin \left (\frac{\pi s}{2} \right ) \Gamma(1-s) \zeta(1-s) \right ) \nonumber \\
&= & \pi^{s} \cos \left (\frac{\pi s}{2} \right ) \Gamma(1-s) \zeta(1-s)+ \cdots \nonumber.
\end{eqnarray}
Evaluate it at $s=-2$, and apply Theorem \ref{thm599.20} to complete the verification.
\end{proof}
\section{Basic Results For The Zeta Quotients $\zeta(5)/\pi^5$}\label{s519}
The fifth odd zeta value has unknown rationality or irrationality properties. 
\begin{conj}  \label{conj519.01} The real numbers 
\begin{multicols}{2}
\begin{enumerate} [font=\normalfont, label=(\roman*)]
\item$ \zeta(5)$,
\item$ 1/\zeta(5)$,
\end{enumerate}
\end{multicols}
are irrationals.
\end{conj}
The decimal expansion $\zeta(5)=1.0 3 6 9 2 7 7 5 5 1 4 3 3 6 9 9\ldots$, appears as sequence $A013663$ in \cite{OEIS}. Assuming this conjecture, the irrationality proofs for various associated real numbers  considered here are derived from the theory of equidistribution of sequences of real numbers, see \cite{KN74} and similar references.
 \begin{thm}\label{thm519.10} Conditional on Conjecture \ref{conj519.01} the real numbers 
\begin{multicols}{2}
\begin{enumerate} [font=\normalfont, label=(\roman*)]
\item$ \zeta(5)/\pi$,
\item$ \pi/\zeta(5)$,
\end{enumerate}
\end{multicols}
are irrationals.
\end{thm}

\begin{proof} (i) On the contrary the number $ \zeta(5)/\pi= A/B$ is a rational number, where $A, B \in\mathbb{N}$ are integers such that $\gcd(A,B)=1$. Proceeds as in the proof of Theorem \ref{thm599.10}, but apply Lemma \ref{lem719.25} to $\sin(t)=\sin(B\zeta(5))\ne 0$ for any integer $B\ne 0$, and Lemma \ref{lem799.72} to complete the argument.
\end{proof}

 \begin{thm}\label{thm519.20} Conditional on Conjecture \ref{conj519.01} the real numbers  
\begin{multicols}{2}
\begin{enumerate} [font=\normalfont, label=(\roman*)]
\item$ \zeta(5)/\pi^2$,
\item$ \pi^2/\zeta(5)$,
\end{enumerate}
\end{multicols}
are irrationals.
\end{thm}

\begin{proof} (i) On the contrary the number $ \zeta(5)/\pi^2= A/B$ is a rational number, where $A, B \in\mathbb{N}$ are integers such that $\gcd(A,B)=1$. Proceeds as in the proof of Theorem \ref{thm599.10}, but apply Lemma \ref{lem719.25} to $\sin(t)=\sin(B\zeta(5)/\pi)\ne 0$ for any integer $B\ne 0$, and Lemma \ref{lem799.72} to complete the argument.
\end{proof}

\begin{thm}\label{thm519.30} Conditional on Conjecture  \ref{conj519.01} the real numbers  
\begin{multicols}{2}
\begin{enumerate} [font=\normalfont, label=(\roman*)]
\item$ \zeta(5)/\pi^3$,
\item$ \pi^3/\zeta(5)$,
\end{enumerate}
\end{multicols}
are irrationals.
\end{thm}

\begin{proof} (i) On the contrary the number $ \zeta(5)/\pi^3= A/B$ is a rational number, where $A, B \in\mathbb{N}$ are integers such that $\gcd(A,B)=1$. Proceeds as in the proof of Theorem \ref{thm599.10}, but apply Lemma \ref{lem719.25} to $\sin(t)=\sin(B\zeta(5)/\pi^2)\ne 0$ for any integer $B\ne 0$, and Lemma \ref{lem799.72} to complete the argument.
\end{proof}

\begin{thm}\label{thm519.40} Conditional on Conjecture \ref{conj519.01} the real numbers  
\begin{multicols}{2}
\begin{enumerate} [font=\normalfont, label=(\roman*)]
\item$ \zeta(5)/\pi^4$,
\item$ \pi^4/\zeta(5)$,
\end{enumerate}
\end{multicols}
are irrationals.
\end{thm}

\begin{proof} (i) On the contrary the number $ \zeta(5)/\pi^4= A/B$ is rational, where $A, B \in\mathbb{N}$ are integers such that $\gcd(A,B)=1$. Proceeds as in the proof of Theorem \ref{thm599.10}, but apply Lemma \ref{lem719.25} to $\sin(t)=\sin(B\zeta(5)/\pi^3)\ne 0$ for any integer $B\ne 0$, and Lemma \ref{lem799.72} to complete the argument.
\end{proof}

\begin{thm}\label{thm519.50} Conditional on Conjecture \ref{conj519.01} the real numbers  
\begin{multicols}{2}
\begin{enumerate} [font=\normalfont, label=(\roman*)]
\item$ \zeta(5)/\pi^5$,
\item$ \pi^5/\zeta(5)$,
\end{enumerate}
\end{multicols}
are irrationals.
\end{thm}

\begin{proof} (i) On the contrary the number $ \zeta(5)/\pi^5= A/B$ is a rational number, where $A, B \in\mathbb{N}$ are integers such that $\gcd(A,B)=1$. Proceeds as in the proof of Theorem \ref{thm599.10}, but apply Lemma \ref{lem719.30} to $\sin(t)=\sin(B\zeta(5)/\pi^4)\ne 0$ for any integer $B\ne 0$, and Lemma \ref{lem799.72} to complete the argument.
\end{proof}

\begin{cor}\label{cor519.13} Conditional on Conjecture \ref{conj519.01}, for any integer $k \in \Z$, the real numbers $ \pi^k$  and $ \zeta(5)$ are linearly independent over the rational numbers $\mathbb{Q}$.
\end{cor}
\begin{proof} Without loss in generality let $k \geq 1$, and assume that the equation
\begin{equation} \label{eq519.78}
 a\pi^k +b\zeta(5)=0
\end{equation}
has a nontrivial rational solution $a,b \in\mathbb{Q}^{\times}$. Proceed to employ the same technique as in the proof of Theorem \ref{thm519.10} to complete the proof.
\end{proof}

\begin{cor}\label{cor519.15} Conditional on Conjecture \ref{conj519.01}, the real number
\begin{equation}
-\zeta^{\prime}(-4)=\frac{24\zeta(5)}{\pi^4}
\end{equation}
is irrational.
\end{cor}
\begin{proof} The functional equation of the zeta function provides an analytic continuation to the entire complex plane plane, see \cite[Theorem 1.6]{IV03}. Thus, its derivative is
\begin{eqnarray}\label{eq599.79}
\zeta^{\prime}(s)
&=&\frac{d}{ds}\left (  2 \pi^{s-1} \sin \left (\frac{\pi s}{2} \right ) \Gamma(1-s) \zeta(1-s) \right ) \nonumber \\
&= & \pi^{s} \cos \left (\frac{\pi s}{2} \right ) \Gamma(1-s) \zeta(1-s)+ \cdots \nonumber.
\end{eqnarray}
Evaluate it at $s=-4$, and apply Theorem \ref{thm519.40} to complete the verification.
\end{proof}
\section{Formulas For Even Zeta Numbers} \label{s190}
The factorization of the sinc function 
\begin{equation}
\frac{\sin 2 \pi x}{2 \pi x}=\prod_{n \geq 1} \left ( 1-\frac{x^2}{n^2}\right )
\end{equation}
is an important analytic tool in the evaluation of the zeta function
\begin{equation}
 \zeta(s)=\sum_{n \geq1} \frac{1}{n^s},
\end{equation}
(and multiple zeta functions), at the even integers $s \geq 2$.

\begin{lem}  {\normalfont (Euler)} \label{thm190.78}A zeta constant at the even integer argument has an exact Euler formula
\begin{equation} \label{eq190.78}
 \zeta(2n)=(-1)^{n+1}\frac{(2 \pi)^{2n} B_{2n}}{2(2n)!} 
\end{equation} 
in terms of the Bernoulli numbers $B_{2n}$, for $n \geq 1$. 
\end{lem}
\begin{proof} Let $s=2n$, and $B_{2n}(\{x\})$ be the $2n$th Bernoulli polynomial, and let the corresponding Fourier series be
\begin{equation}
\sum_{m \geq 1}\frac{\cos(2 \pi m x)}{m^{2n}}=\frac{(-1)^{n+1}(2 \pi)^{2n}}{2} \frac{B_{2n}(\{x\})}{(2n)!},
\end{equation}
where $\{x\}=x-[x]$ is the fractional part function. Evaluating at $x=0$ yields $B_{2n}(\{x\})=B_{2n}(0)=B_{2n}.$
\end{proof}
This is a standard result widely available in the literature, \cite[Theorem 1.4]{IV03}, \cite[p.\ 18]{CO07}, et alii. This formula expresses each zeta constant $\zeta(2n)$ as a rational multiple of $\pi^{2n}$. The formula for the evaluation of the first even zeta constant $\zeta(2)$, known as the Basel problem, was proved by Euler, later it was generalized to all the even integer arguments. Today, there are dozens of proofs, from different technical perspectives, see \cite{CR99}, and \cite[Chapter 6]{SJ03} for an elementary introduction. The first few are 
\begin{multicols}{3}
\begin{enumerate} [font=\normalfont, label=(\roman*)]
\item $ \displaystyle \zeta(2)=\frac{ \pi^2}{6}, $
\item $ \displaystyle \zeta(4)=\frac{ \pi^4}{90}, $
\item $ \displaystyle \zeta(6)=\frac{ \pi^6}{945} ,$
\end{enumerate}
\end{multicols}

et cetera.\\
\section{Formulas For Odd Zeta Numbers} \label{s195}
Currently, the evaluation of a zeta value at an odd integer argument has one or two complicated transcendental power series. A formula for $\zeta(2n+1)$ expresses this constant as a sum of a rational multiple of $\pi^{2n+1}$ and one or more power series. The earliest such series is the Lerch formula
\begin{eqnarray} \label{eq195.23}
\zeta(2n+1)&=&2^{2n} \pi^{2n+1} \sum_{0 \leq k \leq n+1} \frac{(-1)^{k+1}B_{2k}B_{2n+2-2k}}{(2k)! (2n+2-2k)!}-2 \sum_{m \geq 1} \frac{1}{m^{2n+1}( e^{2 \pi m}-1)}\\
&=&a_n\pi^{2n+1}+b_n, \nonumber
\end{eqnarray}for $n \geq 1$. The number $a_n \in \Q$ is rational, but $b_n \in \R$ has unknown arithmetic properties. This is a special case of the Ramanujan series for the zeta function, see \cite[Theorem 1]{GE72}, \cite{GE72B}, \cite{BS17}, et alii. 
The general forms of these formulas are
\begin{equation} 
\zeta(s)=
\begin{cases}
\displaystyle a_n \pi^{4n-1}  -b_n \sum_{n \geq 1} \frac{1}{n^{4n-1}( e^{2 \pi n}-1)}&\text{if $s=4n-1$},\\
\displaystyle a_n \pi^{4n-3}  -b_n \sum_{n \geq 1} \frac{1}{n^{4n-3}( e^{2 \pi n}-1)}-c_n \sum_{n \geq 1} \frac{1}{n^{4n-3}( e^{2 \pi n}+1)}&\text{if $s=4n-3$},
\end{cases}
\end{equation}
where $a_n, b_n, c_n \in \Q$ are rational numbers. The first few are 
\begin{enumerate} [font=\normalfont, label=(\roman*)]
\item $ \displaystyle \zeta(3)=\frac{7 \pi^3}{180}  -2 \sum_{n \geq 1} \frac{1}{n^3( e^{2 \pi n}-1)} $,
\item $ \displaystyle \zeta(5)=\frac{ \pi^5}{294}  -\frac{72}{35} \sum_{n \geq 1} \frac{1}{n^5( e^{2 \pi n}-1)}-\frac{2}{35} \sum_{n \geq 1} \frac{1}{n^5( e^{2 \pi n}+1)}, $
\item $ \displaystyle \zeta(7)=\frac{19 \pi^7}{56700}  -2 \sum_{n \geq 1} \frac{1}{n^7( e^{2 \pi n}-1)} ,$
\end{enumerate}
et cetera. \\

The proof of the generalized formula is based on the associated theory of modular forms. This analysis involves the modular forms such as
\begin{equation}
F_{s}(\tau)= \sum_{n \geq 0} \sigma_{-s}(n)e^{i 2 \pi n \tau }
\end{equation}
where $\sigma_{-s}(n)=\sum_{d \mid n}d^{-s}$ is the sum of divisors function, and 
\begin{equation}
H_{s}(\tau)=(s-1)F_s(\tau)-i2F_s^{}(\tau)
\end{equation} of a complex variable $\tau \in \C$ such that $\Im m (\tau)>0$.

\begin{thm}  {\normalfont (\cite{GE72})} \label{thm195.88} Let $s=2n+1$, $ n \geq 1$. An odd zeta value has a representation as
\begin{equation} \label{eq195.78}
 \zeta(2n)=C_s\pi^s+D_s, 
\end{equation} 
where the first term is defined by

\begin{equation} 
C_s=
\begin{cases}
\displaystyle \frac{2^{2n+1}}{2n(2n+2)} \sum_{0 \leq k \leq n/2}(-1)^k (2n+2-4v) \binom{2n+2}{4k} B_{2v}B_{2n+2-2k}&\text{if $2n+1 \equiv 1 \bmod 4$},\\
\displaystyle \frac{2^{2n}}{(2n+2)!} \sum_{0 \leq k \leq n+1}(-1)^k \binom{2n+2}{2k} B_{2k}B_{2n+2-2k}&\text{if $2n+1 \equiv 3 \bmod 4$}.
\end{cases}
\end{equation}
and the second term
\begin{equation} 
D_s=
\begin{cases}
\displaystyle \frac{2}{s-1}H_s(i)&\text{if $2n+1 \equiv 1 \bmod 4$},\\
\displaystyle 2H_s(i)&\text{if $2n+1 \equiv 3 \bmod 4$}.
\end{cases}
\end{equation}

\end{thm}
The first term $C_s$ is a rational number, but the arithmetic properties of the second term $D_s$ remains unknown. In fact, it is an active area of research in number theory. This analysis is discussed in \cite[Theorem 1]{GE72}, \cite{GE72B}, \cite{BS17}, \cite{LR11}, etc. These formulas express each zeta constant $\zeta(2n+1)$ as a nearly rational multiple of $\pi^{2n+1}$. These analysis are summarized in a compact formula.

\begin{dfn} \label{dfn195.37}  { \normalfont Let $s \geq 2$ be an integer. The $\pi$-representation of the zeta constant $\zeta(s)=\sum_{n \geq 1}n^{-s}$ is defined by the formula
\begin{equation} 
\zeta(s)=
\begin{cases}
\displaystyle a_n \pi^{s}&\text{if $s=4n, 4n+2$},\\
\displaystyle a_n \pi^{s} +b_n&\text{if $s=4n+1,4n+3$},
\end{cases}
\end{equation}
where $ a_n \in \Q$ is a rational number and $b_n \in \R^{\times}$ is a real number.
}
\end{dfn}

\section{Powers of Pi} \label{s185}
The irrationality proof for $\pi$ uses the continued fraction of the tangent function $\tan(x)$, the fact that the numbers $\tan(r)$ are irrationals for any nonzero rational number $r \in \Q^{\times}$, and the value $\arctan(1)=\pi/4$ to indirectly show that the continued fraction 
\begin{equation}
\pi=[3;7,15,1,292,1,1,1,2,1,3,1,14, \ldots]
\end{equation}  
is infinite, see \cite[p.\ 129]{BB04}, \cite{LM97}, \cite{NI47}. Later, simpler versions and new proofs were found by several authors, \cite{NI47}, \cite[p.\ 35]{AZ14}, \cite{SJ05}. 
\begin{thm} \label{thm185.17}   The number $\pi^{s}$ is irrational for any rational power $s \in \Z^{\times}$.
\end{thm}
The nonalgebraic nature of the number $\pi$ can be extended to all the powers by induction, or by other method as done in the second part of the result.
\begin{thm} \label{thm185.27}   The number $\pi^{s}$ is transcendental for any rational power $s \in \Z^{\times}$.
\end{thm}
\begin{proof} (i) Let $s=1$. Assume $\pi$ and its unit $i \pi$ are algebraic over the rational number, and apply the Lindemann-Weierstrass theorem to the exponential $e^{i \pi}=-1$. Since this contradicts the assumption, the number $i \pi$ is transcendental.\\

(ii) For every polynomial $f(x) \in \Z[x]$ the evaluation $f(\pi)\ne 0$ since $\pi$ is nonalgebraic (transcendental). Let $s\in \N$, and assume that $\pi ^s$ is algebraic. Then, there exists a polynomial $g(x) \in \Z[x]$ of degree $\deg g =n$, such that
\begin{eqnarray}
0&=& g(\pi ^s) \\
&=&a_n\left (\pi^{s}\right )^n+a_{n-1}\left (\pi^{s}\right )^{n-1}\cdots+a_1\pi^s+a_0 \nonumber\\
&=&a_n\pi ^{sn}+a_{n-1}\pi^{s(n-1)}\cdots+a_1\pi^s+a_0 \nonumber \\
&=& g_s(\pi)\nonumber, 
\end{eqnarray}
where $g_s(x) \in \Z[x]$ is a polynomial of degree $\deg g_s=sn$. This implies that $\pi$ is algebraic. But, this contradicts the nonalgebraic property of the number $\pi$.       
\end{proof}

Surely, these results can be extended to the rational powers $\pi^s$, where $s \in \Q^{\times}$.

\section{The Irrationality of Some Constants}\label{sec4}
The different analytical techniques utilized to confirm the irrationality, transcendence, and irrationality measures of many constants are important in the development of other irrationality proofs. Some of these results will be used later on.
\begin{thm} \label{thm4.1} The real numbers \(\pi, \;\zeta (2),\text{ and } \zeta (3)\) are irrational numbers.
\end{thm}
The various irrationality proofs of these numbers are widely available in the open literature. These technique are valuable tools in the theory of irrational numbers, refer to \cite{AR79}, \cite{BF79}, \cite{HD01}, \cite{SJ05}, and others.\\

\begin{thm} \label{thm4.2} For any fixed \(n \in \mathbb{N}\), and the nonprincipal character $\chi \mod 4$, the followings statements are valid.
\begin{enumerate} [font=\normalfont, label=(\roman*)]
\item The real number \(\displaystyle \zeta (2n)=\frac{(-1)^{n+1}2^{2n}B_{2n}}{(2n)!}\pi^{2n}\) is a transcendental number,
\item The real number \(\displaystyle L(2n+1,\chi )=\frac{(-1)^n E_{2n}}{2^{2n+2}(2n)!}\pi^{2n+1}\) is a transcendental number, where \(B_{2n}\text{ and } E_{2n}\) are the Bernoulli and Euler numbers respectively.
\end{enumerate}
\end{thm}

\begin{proof} Apply Theorem \ref{thm185.27} or the Lindemann-Weierstrass theorem to the transcendental number $\pi$. \end{proof}

The first few nonvanishing Bernoulli numbers are these:
\begin{multicols}{3}
\begin{enumerate} [font=\normalfont, label={}]
\item $B_0=1$, \item $B_1=\frac{-1}{2}$, \item $B_2=\frac{1}{6}$, \item $B_4=\frac{-1}{30}$,
\item $B_6=\frac{1}{42}$, \item $B_8=\frac{-1}{30}$,$\ldots.$ 
\end{enumerate}
\end{multicols}
And the first few nonvanishing Euler numbers are these:
\begin{multicols}{3}
\begin{enumerate} [font=\normalfont, label={}]
\item $E_0=1$, \item $E_2=-1$,\item $E_4=5$,\item $E_6=-161$, \item $E_8=1385$, \item $E_{10}=-50521 \ldots$.    
\end{enumerate}
\end{multicols}
The generalization of these results to number fields is discussed in \cite{ZD86}, and related literature.

\begin{thm} \label{thm4.3} {\normalfont (Klinger)} Let $\mathcal{K}$ be a number field extension of degree $k=[\mathcal{K} : \mathbb{Q}]$, and discriminant $D=disc (\mathcal{K})$. Then

\begin{enumerate} [font=\normalfont, label=(\roman*)]
\item  If $D>0$, the number field is totally real and $\displaystyle \zeta_{\mathcal{K}}(2n)=r_k\frac{\pi^{2nk}}{\sqrt{D}}$, where $n \geq 1$, and $r_k \in \mathbb{Q}$ is a transcendental number.
\item  If $D<0$, the number field is totally complex and $\displaystyle \zeta_{\mathcal{K}}(1-2n)=r_k$, where $n \geq 1$, and $r_k \in \mathbb{Q}$.
\end{enumerate}
\end{thm}

\begin{proof} (i) If $ \zeta_{\mathcal{K}}(2n)$ is algebraic, then there exists a rational polynomial $f(x) \in \Z[x]$ of even degree $\deg f =2d$ such that $f( \zeta_{\mathcal{K}}(2n))=0$. But this is false. It contradicts the nonalgebraic property of the real number $\pi$. \end{proof}

\section{Problems}
\subsection{Nonalgebraic Numbers}
\begin{exe} 
{\normalfont Prove that $\pi$ is nonalgebraic implies that $\pi^r$ is nonalgebraic for any rational number $r \in \Q^{\times}$.
}
\end{exe}

\begin{exe} 
{\normalfont Prove that $\alpha$ is nonalgebraic implies that $\alpha^r$ is nonalgebraic for any rational number $r \in \Q^{\times}$.
}
\end{exe}
\begin{exe} 
{\normalfont Prove that $\pi^n/\sqrt[k]{2}$ is nonalgebraic for any pair $k,n \in \Z^{\times}$.
}
\end{exe}
\begin{exe} 
{\normalfont Is the real number $e+\pi$ algebraic, is there a rational polynomial $f(x) \in \mathbb{Z}[x]$ such that $f(e+\pi)=0$?
}
\end{exe}
\begin{exe} 
{\normalfont Is the real number $e/\pi$ algebraic, is there a rational polynomial $f(x) \in \mathbb{Z}[x]$ such that $f(e/\pi)=0$?
}
\end{exe}
\subsection{Rational/Irrational Numbers}
\begin{exe} 
{\normalfont Prove or disprove that $\zeta(3)=a\pi^{3}+b$, where $a,b \in \Q^{\times}$.
}
\end{exe}
\begin{exe} 
{\normalfont Prove or disprove that $\zeta(5)=a\pi^{5}+b$, where $a,b \in \Q^{\times}$.
}
\end{exe}
\begin{exe} 
{\normalfont Prove or disprove that $\zeta(2n+1)=a\pi^{2n+1}+b$, where $a,b \in \Q^{\times}$, and $n \geq 1$.
}
\end{exe}
\begin{exe} 
{\normalfont It is known that $\zeta(-1)=1/12.$ Prove or disprove that $\zeta(-3)$ rational/irrational.
}
\end{exe}

\subsection{Bounded/Unbounded Partial Quotients}
\begin{exe} 
{\normalfont Does the real number $e/\pi=[a_0,a_1,a_2, \ldots]$ have unbounded partial quotients $a_n \in \mathbb{N}$? } \end{exe}

\begin{exe} 
{\normalfont Is the real number $e+\pi$ algebraic, is there a rational polynomial $f(x) \in \mathbb{Z}[x]$ such that $f(e+\pi)=0$?
}
\end{exe}

\begin{exe} 
{\normalfont Does the real number $e+\pi=[a_0,a_1,a_2, \ldots]$ have unbounded partial quotients $a_n \in \mathbb{N}$?
}
\end{exe}

\subsection{Normal/Nonnormal Numbers}

\begin{exe} 
{\normalfont Prove or disprove whether or not $\pi$ is a normal number.
}
\end{exe}
\begin{exe} 
{\normalfont Prove or disprove whether or not $\pi^2$ is a normal number.
}
\end{exe}
\begin{exe} 
{\normalfont Prove or disprove whether or not $e$ is a normal number.
}
\end{exe}
\begin{exe} 
{\normalfont Prove or disprove whether or not $e+\pi$ is a normal number.
}
\end{exe}

\currfilename.\\

\end{document}